\RequirePackage{fix-cm}
\documentclass[smallextended]{svjour3}       
\smartqed  
\usepackage{threeparttable}
\usepackage{graphicx}
\usepackage{multirow}%
\usepackage{amsmath,amssymb,amsfonts}%
\usepackage[numbers]{natbib}
\usepackage{mathrsfs}%
\usepackage[title]{appendix}%
\usepackage{xcolor}%
\usepackage{textcomp}%
\usepackage{manyfoot}%
\usepackage{booktabs}%
\usepackage{marvosym}
\usepackage{bm}
\usepackage{algorithm}%
\usepackage{algorithmicx}%
\usepackage{algpseudocode}%
\usepackage{listings}%
\usepackage{float,subcaption,geometry,pdfpages,tabularx}
\usepackage{array}
\usepackage{threeparttable}
\numberwithin{theorem}{section}
\numberwithin{lemma}{section}
\numberwithin{definition}{section}

\begin{document}

\title{
A Multi-Invariant Preserving Discrete Gradient Methods
}


\author{Haodong Pu      \and
        Maohua Ran
}


\institute{
           Maohua Ran (\Letter)\at
               School of Mathematical Sciences, Sichuan Normal University, Chengdu 610066, China\\
              \email{maohuaran@163.com} 
              \and
              Haodong Pu \at
              School of Mathematical Sciences, Sichuan Normal University, Chengdu 610066, China\\
              \email{puhaodong10@163.com}   
}

\date{Received: date / Accepted: date}

\maketitle

\begin{abstract}
This work introduces a novel structure-preserving methods for conservative systems based on a predictor-corrector strategy. The framework applies a discrete gradient correction to predictions generated by explicit one-step or multi-step schemes, which preserves nonlinear invariants while maintaining the accuracy order of the original predictor. This approach extends naturally to problems requiring simultaneous conservation of multiple invariants. Under mild conditions, conservation properties, solvability, numerical accuracy, and stability are established. Long-term numerical simulations on Lotka-Volterra systems, sine-Gordon equations, rigid body dynamics, and Kepler problems demonstrate improved robustness and conservation properties compared to existing projection and relaxation methods.
\end{abstract}

\maketitle
\section{Introduction}\label{sec1}
Structure-preserving numerical methods are designed to maintain key geometric and physical structures-such as symmetries, invariants, and symplecticity-inherent in differential equations. Over the past few decades, significant theoretical and computational advances have been made in this field. Among these developments, the preservation of invariants plays a particularly vital role, as it is crucial for simulating long-term dynamics accurately in scientific and engineering applications \cite{sharma2020review}.

This work focuses on conservative systems of the form
\begin{equation}\label{ode}
\begin{cases}
y'(t)=f(y(t)),\quad t\in [0,T], \\
y(0)=y_0,
\end{cases}
\end{equation}
where $f:\mathbb{R}^d\to \mathbb{R}^d$ is sufficiently smooth, $y_0\in \mathbb{R}^d$ is the initial value. The system admits a first integral $H(y)$ such that
$$H(y(t)) \equiv H(y_0),\quad t\in[0,T].$$
Consequently, the solution trajectory is confined to the invariant manifold 
$\mathcal{M} = \{ y \in \mathbb{R}^d \mid H(y) = H(y_0)\}$,
a geometric constraint that structure-preserving numerical methods must respect by ensuring the discrete solution ${y_n}$ satisfies $H(y_n) = H(y_0)$.

Significant progress has been made in structure-preserving numerical methods, with Runge-Kutta (RK) methods preserving arbitrary linear invariants \cite{Cooper} and symplectic RK methods conserving quadratic invariants, though no RK method can preserve cubic or higher-degree polynomial invariants \cite{Hairer}. This limitation has spurred the development of specialized algorithms, including for Hamiltonian systems: variational integrators based on discrete variational principles \cite{Marsden}, Hamiltonian boundary value methods that exactly preserve the Hamiltonian \cite{Brugnano}, energy-momentum integrators for mechanical systems \cite{Simo}, and discrete Hamiltonian methods using discrete directional derivatives \cite{gonzalez}. Furthermore, techniques originally developed for gradient flows-invariant energy quadratization \cite{Yang} and scalar auxiliary variable approaches \cite{SJ}-have been extended to construct energy-preserving schemes for Hamiltonian systems \cite{Liu, Fu}. For general invariants, widely used methods include the discrete gradient method \cite{mclachlan}, Lagrange multiplier method \cite{Gear}, and projection method \cite{Hairer2000}; however, these typically require solving large-scale equation systems with considerable computational costs, and the simultaneous preservation of multiple invariants remains a fundamental challenge in the field.
To address this, efficient strategies such as the projected method \cite{Stuart} and the relaxed Runge-Kutta (RRK) method \cite{del, Ketcheson} have been introduced, which reduce computational complexity by solving only one scalar equation per time step. Despite these improvements, three practical challenges remain in the design of invariant-preserving integrators:
\begin{itemize}
\item[$\cdot$] \textbf{Low-order accuracy}: Classical discrete gradient methods are generally limited to first- or second-order accuracy, since constructing higher-order discrete gradients that simultaneously satisfy the discrete chain rule presents substantial analytical difficulties.
\item[$\cdot$] \textbf{Poor large-step robustness}: Projection and relaxation-based methods (such as RRK) exhibit robustness issues when the predictor direction is nearly orthogonal to $\nabla H$, or when the time step exceeds stability limits.
\item[$\cdot$] \textbf{Non-modular multi-invariant conservation}: Most existing methods are designed for single invariants. For systems with multiple invariants, traditional projection techniques require iterative projections or more intricate geometric treatments, while multiple-relaxation Runge-Kutta \cite{Biswas} and multiple-relaxation exponential Runge-Kutta methods \cite{LDF} are inherently constrained by their underlying Runge-Kutta or exponential Runge-Kutta frameworks. Moreover, the existence of their relaxation parameters-and consequently, of the numerical solution-is guaranteed only under sufficiently strict time-step restrictions.
\end{itemize}

To overcome these limitations, this work introduces a discrete gradient correction framework with three core features:
\begin{itemize}
    \item[$\cdot$] \textbf{Compatibility}: Enables decoupling of numerical accuracy from structure preservation, allowing any explicit one-step or multi-step method to serve as predictor while maintaining invariant conservation without order reduction through the correction procedure.    
    \item[$\cdot$] \textbf{Well-posedness}: Theoretically establishes existence of the numerical solution for arbitrary time steps, while guaranteeing solution uniqueness under mild assumptions for sufficiently small steps.    
    \item[$\cdot$] \textbf{Scalability}: The approach naturally generalizes to conserve multiple invariants simultaneously without modifying the core algorithmic structure.
\end{itemize}

The remainder of this paper is organized as follows. 
Section~\ref{sec2} introduces basic notations, discrete gradients, and preliminary lemmas; 
Section~\ref{sec3} presents discrete gradient correction methods for single invariants, analyzing convergence order, conservation, well-posedness, and stability; 
Section~\ref{sec4} extends these methods to systems with multiple invariants and provides theoretical analysis; 
Section~\ref{sec5} demonstrates effectiveness through numerical experiments; 
The conclusion summarizes the main contributions.

\section{Preliminaries}\label{sec2}
This section establishes the basic notation, introduces the concept of discrete gradients, and presents several preliminary results that will be used throughout the paper.

Let \(\mathbf{w}\) and \(\mathbf{v}\) denote vectors in \(\mathbb{R}^d\). The Euclidean inner product and norm are defined as \(\mathbf{w} \cdot \mathbf{v}\) and \(\|\mathbf{w}\| = \sqrt{\mathbf{w} \cdot \mathbf{w}}\), respectively.
\begin{definition}[Discrete gradient,\cite{mclachlan}]\label{df1}
For a continuously differentiable function $H$, a discrete gradient is a continuous mapping $\overline{\nabla}H$ satisfying 
\begin{equation*}
        \overline{\nabla}H(x,y)\cdot(y-x)=H(y)-H(x)~\mbox{and}~
        \overline{\nabla}H(y,y)=\nabla H(y),~\forall x,y\in\mathbb{R}^d.
\end{equation*}
\end{definition}

Common discrete gradient formulations include the following:
\begin{enumerate}
    \item[$\cdot$]\textit{Gonzalez discrete gradient} \cite{gonzalez}:
    \begin{align}\label{DG1}
        \overline{\nabla}H(x,y)=\nabla H\left(\frac{x+y}{2}\right)+\frac{H(y)-H(x)-\nabla H\left(\frac{x+y}{2}\right)\cdot(y-x)}{\|x-y\|^2}(y-x).
    \end{align}    
    \item[$\cdot$]  \textit{Mean value discrete gradient} \cite{Harten}:
    \begin{align}\label{DG2}
        \overline{\nabla}H(x,y)=\int^1_0\nabla H\big((1-s)x+sy\big)\mathrm{d}s.
    \end{align}    
    \item[$\cdot$]  \textit{Coordinate increment discrete gradient} \cite{Itoh}:
    \begin{align}\label{DG3}
        \overline{\nabla} H(x,y)=\begin{pmatrix}
            \dfrac{H(y_1,x_2,\ldots,x_n)-H(x)}{y_1-x_1} \\
            \dfrac{H(y_1,y_2,x_3,\ldots,x_n)-H(y_1,x_2,\ldots,x_n)}{y_2-x_2} \\
            \vdots \\
            \dfrac{H(y)-H(y_1,\ldots,y_{n-1},x_n)}{y_n-x_n}
        \end{pmatrix}.
    \end{align}
\end{enumerate}

\begin{lemma}[Brouwer's fixed point theorem, \cite{ciarlet}]\label{BF}
Let $U$ be a compact convex subset of a finite-dimensional normed vector space. Then, continuous mapping $g: U \to U$ admits at least one fixed point.
\end{lemma}
\begin{lemma}[Banach fixed point theorem, \cite{ciarlet}]\label{BAF}
Let $(X,d)$ be a complete metric space. Any contraction mapping $f:X\to X$ possesses exactly one fixed point $x^{*}\in X$.  Furthermore, for arbitrary $x_0\in X$, the sequence 
$\{x_n\}_{n=0}^{\infty}$ defined by
    \[x_{n+1}=f(x_n),\quad n\geq 0,\]
    converges to $x^{*}$ as $n\to\infty$, satisfying the error estimate:
    \[\|x_n-x^{*}\|\leq \frac{d(f(x_0),x_0)}{1-k}\,k^n, \quad n\geq 0,\]
    where $k~(0<k<1)$ denotes the contraction constant.
\end{lemma}

\section{Invariant-Preserving Methods}\label{sec3}
Let \( N \) be a positive integer, \( h = T/N \) the step size, and \( t_n = nh \) for \( n = 0, 1, \dots, N \). Consider an explicit \( m \)-step method producing a predictor:
\begin{align}\label{DGC1}
    \overline{y}_{n+1} = \phi_h(y_n, y_{n-1}, \dots, y_{n-m}).
\end{align}
The discrete gradient correction (DGC) method for solving \eqref{ode} is then defined by the correction step:
\begin{align}\label{DGC}
        y_{n+1} = \overline{y}_{n+1} + \frac{H(y_0) - H(\overline{y}_{n+1})}{\left\| \overline{\nabla}H(\overline{y}_{n+1},y_{n+1}) \right\|^2} \overline{\nabla}H(\overline{y}_{n+1},y_{n+1}).
\end{align}

\begin{remark}
 When \( m = 0 \), the predictor \eqref{DGC1} reduces to an explicit one-step method. This special case finds frequent application in practical computations and will be demonstrated in subsequent numerical examples.
\end{remark}

\begin{remark}

 Unlike standard projection methods, which are merely geometric post-processing operations, the DGC method preserves the structure of the underlying discrete conservation laws.

\end{remark}

Since the correction equation \eqref{DGC} is implicit, we analyze the existence and uniqueness of its solution. For a fixed predictor $\overline{y}_{n+1}$, define the iteration function
\begin{align}\label{DGCb}
 g(\overline{y}_{n+1}, \mathbf{v}) = \overline{y}_{n+1} + \frac{H(y_0) - H(\overline{y}_{n+1})}{\left\| \overline{\nabla} H(\overline{y}_{n+1}, \mathbf{v}) \right\|^2} \overline{\nabla} H(\overline{y}_{n+1}, \mathbf{v}).
\end{align}
A solution $y_{n+1}$ of \eqref{DGC} is then identified as a fixed point of $G$, satisfying $y_{n+1} = g(\overline{y}_{n+1},y_{n+1})$.

\begin{theorem}[Existence]\label{T31}
If there exists a constant $C > 0$ such that $\left\| \overline{\nabla}H(\overline{y}_{n+1},v) \right\| > C$ for all $v$ in some neighborhood of $\overline{y}_{n+1}$, 
 then the DGC method \eqref{DGC1}-\eqref{DGC} has at least one solution.
\end{theorem}

\begin{proof}
Define 
\[U=\left\{v|\left\|v-\overline{y}_{n+1}\right\|\le \frac{1}{C}|H(y_0)-H(\overline{y}_{n+1})|\right\}.\]
It is straightforward to verify that $U$ is compact and convex. Noticing that  
    \[\left\|g(\overline{y}_{n+1},v)-\overline{y}_{n+1}\right\|=\frac{1}{\left \| \overline{\nabla}H(\overline{y}_{n+1},v) \right \|}|H(y_0)-H(\overline{y}_{n+1})|\le \frac{1}{C}|H(y_0)-H(\overline{y}_{n+1})|,\]
    which implies $g(\overline{y}_{n+1}, v) \in U$. By Definition~\ref{df1}, the discrete gradient $\overline{\nabla}H(\overline{y}_{n+1}, v)$ is continuous in $v$. Therefore, $g(\overline{y}_{n+1}, v)$ is continuous on $U$. Applying Brouwer's fixed-point theorem (Lemma~\ref{BF}) to the continuous function $g(\overline{y}_{n+1}, \cdot)$ mapping the compact convex set $U$ into itself, we conclude that $g$ has at least one fixed point in $U$. This completes the proof.
\end{proof}

\begin{theorem}[Uniqueness]\label{T32}
Under the assumptions of Theorem~\ref{T31}, assume that the discrete gradient is locally Lipschitz continuous: there exists a constant \( {L_0} > 0 \) such that
\[
\left\|\overline{\nabla}H(\overline{y}_{n+1},w) - \overline{\nabla}H(\overline{y}_{n+1},v)\right\| \le {L_0} \left\| w - v \right\|
\]
for all \( v, w \) in a neighborhood of \( \overline{y}_{n+1} \). Then, for sufficiently small step size \( h \), the DGC method \eqref{DGC1}-\eqref{DGC} has a unique solution.
\end{theorem}

\begin{proof}
A standard estimation yields:
\begin{align*}
    &\left\| g(\overline{y}_{n+1},w) - g(\overline{y}_{n+1},v) \right\|\\
    =&~ \left| H(y_0) - H(\overline{y}_{n+1}) \right| \cdot \left\| \frac{\overline{\nabla}H(\overline{y}_{n+1},w)}{\left\| \overline{\nabla}H(\overline{y}_{n+1},w) \right\|^2} - \frac{\overline{\nabla}H(\overline{y}_{n+1},v)}{\left\| \overline{\nabla}H(\overline{y}_{n+1},v) \right\|^2} \right\| \\
     =&~ \frac{\left| H(y_0) - H(\overline{y}_{n+1}) \right|}{\left\| \overline{\nabla}H(\overline{y}_{n+1},w) \right\|\cdot\left\| \overline{\nabla}H(\overline{y}_{n+1},v) \right\|} \cdot \left\| \overline{\nabla}H(\overline{y}_{n+1},w)-\overline{\nabla}H(\overline{y}_{n+1},v) \right\| \\
     \leq &~ \frac{L_0}{C^2} \left| H(y_0) - H(\overline{y}_{n+1}) \right| \left\| w - v \right\|.
\end{align*}
Since the predictor \( \overline{y}_{n+1} \) is generated by a convergent explicit method, we have \( \overline{y}_{n+1} \to y(t_{n+1}) \) as \( h \to 0 \). By the continuity of \( H \), it follows that \( H(\overline{y}_{n+1}) \to H(y_0) \). It implies that
\[
0 < \frac{L_0}{C^2} \left| H(y_0) - H(\overline{y}_{n+1}) \right| < 1
\]
for sufficiently small $h$. By the Banach fixed-point theorem (Lemma~\ref{BAF}), there exists a unique $w$ satisfying $g(\overline{y}_{n+1},w) = w$. Therefore, the DGC method \eqref{DGC1}-\eqref{DGC} has a unique solution. This completes the proof.
\end{proof}


\begin{theorem}[Conservation] 
The DGC method \eqref{DGC1}-\eqref{DGC} preserves the invariant $H$.
\end{theorem}

\begin{proof}
By taking the inner product of both sides of the correction scheme
\[
y_{n+1} - \overline{y}_{n+1} = \frac{H(y_0) - H(\overline{y}_{n+1})}{\left\| \overline{\nabla}H(\overline{y}_{n+1},y_{n+1}) \right\|^2} \overline{\nabla}H(\overline{y}_{n+1},y_{n+1})
\]
with $\overline{\nabla}H(\overline{y}_{n+1},y_{n+1})$, we obtain the following identity:
\begin{align*}
    H(y_{n+1}) - H(\overline{y}_{n+1}) &= H(y_0) - H(\overline{y}_{n+1}).
\end{align*}
That is,
\begin{align*}
    H(y_{n+1}) = H(y_0).
\end{align*}
This completes the proof.
\end{proof}

\begin{theorem}[Consistency]
Under the assumptions of Theorem~\ref{T31}, if the one-step/multi-step method \eqref{DGC1} has a local truncation error of order $p$, then the DGC method \eqref{DGC1}-\eqref{DGC} also possesses a local truncation error of order $p$.
\end{theorem}

\begin{proof}
Since the local truncation error of the one-step/multi-step method \eqref{DGC1}  is $\mathcal{O}(h^{p+1})$, we obtain
\begin{align*}
    H(y_0) - H(\overline{y}_{n+1}) 
        &= H(y(t_{n+1})) - H(\overline{y}_{n+1}) \\
        &= (\overline{y}_{n+1} - y(t_{n+1})) \cdot \nabla H\left( \overline{y}_{n+1} + \xi(y(t_{n+1}) - \overline{y}_{n+1}) \right) \\
        &= \mathcal{O}(h^{p+1}),
\end{align*}
where $\xi \in (0,1)$. Consequently,
\begin{align*}
    \left\| y_{n+1} - y(t_{n+1}) \right\| 
        &\leq \left\| \overline{y}_{n+1} - y(t_{n+1}) \right\| + \left\| g(\overline{y}_{n+1}, y_{n+1}) - \overline{y}_{n+1} \right\| \\
        &\leq \mathcal{O}(h^{p+1}) + \frac{\left| H(y_0) - H(\overline{y}_{n+1}) \right|}{C} \\
        &= \mathcal{O}(h^{p+1}).
\end{align*}
This completes the proof.
\end{proof}

\begin{theorem}[Stability] 
Assume \eqref{DGC1} is a one-step method that has the form
\begin{align*}
\overline{y}_{n+1} = y_n + h \varphi_h(y_n),
\end{align*}
and define
\begin{align*}
\psi_h(y_n) = \varphi_h(y_n) \cdot \overline{\nabla}H(\overline{y}_{n+1}, y_n).
\end{align*}
If $\varphi_h, \psi_h$ are Lipschitz continuous with constants ${L_{\varphi}}$ and ${L_{\psi}}$, respectively, $\psi_h$ is uniformly bounded above by a positive constant $C_1$ and the discrete gradient $\overline{\nabla}H(x,y)$ is Lipschitz continuous in both arguments $x$ and $y$ with constant $L_0$, then the DGC method \eqref{DGC1}-\eqref{DGC} is initial-value stable under the assumptions of Theorem~\ref{T31}.
\end{theorem}

\begin{proof}
Let $y(t)$ and $z(t)$ be exact solutions of \eqref{ode} with initial values $y_0$ and $z_0$, and $\{y_n\}$, $\{z_n\}$ their numerical solutions, respectively. Define
\begin{align*}
e_n = y_n - z_n, \quad \overline{e} _n = \overline{y}_n - \overline{z}_n.
\end{align*}
By the invariant-preserving property of \eqref{DGC}, we obtain
\begin{align*}
    H(y_0) - H(\overline{y}_{n+1}) 
        &= H(y_n) - H(\overline{y}_{n+1}) \\
        &= (y_n - \overline{y}_{n+1}) \cdot \overline{\nabla}H(\overline{y}_{n+1}, y_n) \\
        &= -h \varphi_h(y_n) \cdot \overline{\nabla}H(y_n + h\varphi_h(y_n), y_n) \\
        &= -h \psi_h(y_n),
\end{align*}
and
\begin{align*}
    H(z_0) - H(\overline{z}_{n+1}) =-h \psi_h(z_n).
\end{align*}
Thus, we have
\begin{align*}
    y_{n+1} &= \overline{y}_{n+1}-\frac{h \psi_h(y_n)}{\left\| \overline{\nabla}H(\overline{y}_{n+1}, y_{n+1}) \right\|^2} \overline{\nabla}H(\overline{y}_{n+1}, y_{n+1}), 
\end{align*}
and
\begin{align*}
    z_{n+1} &= \overline{z}_{n+1}-\frac{h \psi_h(z_n)}{\left\| \overline{\nabla}H(\overline{z}_{n+1}, z_{n+1}) \right\|^2} \overline{\nabla}H(\overline{z}_{n+1}, z_{n+1}).
\end{align*}
Consequently,
\begin{align}
    \left\| e_{n+1} \right\| 
        &\leq \left\| \overline{e}_{n+1} \right\| + h \left\| \frac{\psi_h(z_n) - \psi_h(y_n)}{\left\| \overline{\nabla}H(\overline{y}_{n+1}, y_{n+1}) \right\|^2} \overline{\nabla}H(\overline{y}_{n+1}, y_{n+1}) \right\| \nonumber \\
        &\quad + h \left\| \psi_h(z_n) \right\| \left\| \frac{\overline{\nabla}H(\overline{y}_{n+1}, y_{n+1})}{\left\| \overline{\nabla}H(\overline{y}_{n+1}, y_{n+1}) \right\|^2} - \frac{\overline{\nabla}H(\overline{z}_{n+1}, z_{n+1})}{\left\| \overline{\nabla}H(\overline{z}_{n+1}, z_{n+1}) \right\|^2} \right\| \nonumber \\
        &\leq \left\| \overline{e}_{n+1} \right\| + \frac{h}{C} \left\| \psi_h(z_n) - \psi_h(y_n) \right\| + \frac{h C_1}{C^2} \left\| \overline{\nabla}H(\overline{y}_{n+1}, y_{n+1}) - \overline{\nabla}H(\overline{z}_{n+1}, z_{n+1}) \right\| \nonumber \\
        &\leq \left\| \overline{e}_{n+1} \right\| + \frac{h {L_{\psi}}}{C} \left\| e_n \right\| + \frac{h C_1 L_0}{C^2} \left( \left\| \overline{e}_{n+1} \right\| + \left\| e_{n+1} \right\| \right). \label{s1}
\end{align}
Noticing that
\begin{align}
    \left\| \overline{e}_{n+1} \right\| 
        \leq \left\| e_n \right\| + h \left\| \varphi_h(y_n) - \varphi_h(z_n) \right\|
        \leq (1 + h {L_{\varphi}}) \left\| e_n \right\|, \label{s2}
\end{align}
and substituting \eqref{s2} into \eqref{s1} yields
\begin{align*}
\left( 1 - \frac{h C_1 L_0}{C^2} \right) \left\| e_{n+1} \right\| \leq \left[ \frac{h {L_{\psi}}}{C} + \left( 1 + \frac{h C_1 L_0}{C^2} \right) (1 + h {L_{\varphi}}) \right] \left\| e_n \right\|.\label{s2b}
\end{align*}
As a result, when $h < h_0 := \min \left\{ 1, \frac{C^2}{2 C_1 L_0} \right\}$, we have

\begin{align*}
    \left\| e_{n} \right\| 
        &\leq \frac{1 + h C_3}{1 - h C_2} \left\| e_{n-1} \right\| 
        \leq \left( 1 + h (2 C_2 + 2 C_3) \right) \left\| e_{n-1}\right\|\leq e^{T C_4} \left\| e_0 \right\|,
\end{align*}
where
\[
C_2 = \frac{C_1 L_0}{C^2}, ~C_3 = \frac{{L_{\psi}}}{C} + \frac{C_1 L_0}{C^2} + {L_{\varphi}} + \frac{C_1 L_0 {L_{\varphi}}}{C^2},~C_4 = 2 C_2 + 2 C_3.
\]
That is,
\begin{align*}
    \left\| y_n-z_n \right\| 
        &\leq (1 + h C_4) \left\| e_{n-1} \right\|
        \leq e^{T C_4} \left\| y_0-z_0 \right\|.
\end{align*}
It means that the DGC method \eqref{DGC1}-\eqref{DGC} is initial-value stable. This completes the proof.
\end{proof}

\section{Multi-Invariant Generalization}\label{sec4}

The DGC method can be naturally extended to conservative systems \eqref{ode} possessing $k$ invariants:
\[
H_i(y(t)) \equiv H_i(y_0), \quad i = 1, 2, \cdots, k, \quad t \in [0, T].
\]
The corresponding DGC method for multiple invariant preservation is given by
\begin{align}\label{DGCM}
    \begin{cases}
        \overline{y}_{n+1} = \phi_h(y_n, y_{n-1}, \cdots, y_{n-m}), \\
        y_{n+1} = g(\overline{y}_{n+1}, y_{n+1}),
    \end{cases}
\end{align}
where the correction function $g$ is defined as
\begin{align*}
    g(\overline{y}_{n+1}, y_{n+1}) = \overline{y}_{n+1} + \sum_{i=1}^{k} \lambda_i \, \overline{\nabla}H_i(\overline{y}_{n+1}, y_{n+1}),
\end{align*}
with the coefficient vector $\boldsymbol{\lambda} = (\lambda_1, \dots, \lambda_k)^\top$ obtained by solving the linear system
\begin{align}\label{DGCMb}
    A \boldsymbol{\lambda} = \mathbf{b}.
\end{align}
Here, $A$ is the symmetric positive semi-definite matrix
\begin{align*}
    A = \left( \overline{\nabla}H_i(\overline{y}_{n+1}, y_{n+1}) \cdot \overline{\nabla}H_j(\overline{y}_{n+1}, y_{n+1}) \right)_{k \times k},
\end{align*}
and $\mathbf{b}$ is the vector
\begin{align*}
    \mathbf{b} = \left( H_1(y_0) - H_1(\overline{y}_{n+1}), \dots, H_k(y_0) - H_k(\overline{y}_{n+1}) \right)^\top.
\end{align*}

\begin{theorem}[Existence]
If all eigenvalues of the matrix $A$ are positive and uniformly bounded below by a positive constant $C$, then the DGC method \eqref{DGCM} admits at least one solution.
\end{theorem}

\begin{proof}
The continuity of the discrete gradient implies that $A(\overline{y}_{n+1},v)$ is continuous in $v$. Given that \(A\) is invertible, \(A^{-1}\) is also continuous. Consequently, $\lambda = A^{-1}b(\overline{y}_{n+1})$ depends continuously on $v$, and so does $g(\overline{y}_{n+1},v)$. 

From the definition of $g$, we have
\begin{align*}
    \left\| g(\overline{y}_{n+1},v) - \overline{y}_{n+1} \right\|^2 = \boldsymbol{\lambda}^\top A \boldsymbol{\lambda}.
\end{align*}
Substituting $\boldsymbol{\lambda} = A^{-1} b(\overline{y}_{n+1})$ yields
\begin{align*}
    \left\| g(\overline{y}_{n+1},v) - \overline{y}_{n+1} \right\|^2 = 
    \mathbf{b}^\top (\overline{y}_{n+1}) A^{-1} \mathbf{b}(\overline{y}_{n+1}).
\end{align*}
Since $A$ is symmetric positive definite with all eigenvalues greater than $C > 0$, the largest eigenvalue of $A^{-1}$ is bounded by $C^{-1}$. By the Rayleigh quotient principle:
\begin{align*}
    \mathbf{b}(\overline{y}_{n+1})^\top A^{-1} \mathbf{b}(\overline{y}_{n+1})
    \leq C^{-1} \| \mathbf{b}(\overline{y}_{n+1}) \|^2.
\end{align*}
Therefore,
\begin{align*}
    \left\| g(\overline{y}_{n+1},v) - \overline{y}_{n+1} \right\| \leq 
    C^{-1/2} \|\mathbf{b}(\overline{y}_{n+1}) \|.
\end{align*}
Now define
\[
U = \left\{ v \middle| \|v - \overline{y}_{n+1}\| \leq C^{-1/2} \|\mathbf{b}(\overline{y}_{n+1})\| \right\}.
\]
The above inequality shows that $g(\overline{y}_{n+1}, \cdot)$ maps $U$ into itself. As $U$ is a closed ball (hence compact and convex), by Lemma~\ref{BF} (Brouwer fixed-point theorem), there exists $v^* \in U$ such that $g(\overline{y}_{n+1}, v^*) = v^*$, which provides a solution to the DGC method \eqref{DGCM}. This completes the proof.
\end{proof}

\begin{theorem}[Consistency]
Assume that all eigenvalues of the matrix $A$ are positive and uniformly bounded below by a positive constant $C$. If the prediction method $\phi_h$ has a local truncation error of order $p$, then the DGC method \eqref{DGCM} also achieves a local truncation error of order $p$.
\end{theorem}

\begin{proof}
If $\phi_h$ has a local truncation error of order $p$, then
\[
\|\overline{y}_{n+1} - y(t_{n+1})\| = \mathcal{O}(h^{p+1}), \quad \|\mathbf{b}(\overline{y}_{n+1})\| = \mathcal{O}(h^{p+1}).
\]
Consequently,
\begin{align*}
    \|y_{n+1} - y(t_{n+1})\| & \leq \|\overline{y}_{n+1} - y(t_{n+1})\| + \|g(\overline{y}_{n+1}, y_{n+1}) - \overline{y}_{n+1}\| \\
    & \leq \|\overline{y}_{n+1} - y(t_{n+1})\| + C^{-1/2} \|\mathbf{b}(\overline{y}_{n+1})\| \\
    & = \mathcal{O}(h^{p+1}).
\end{align*}
This establishes the desired result.
\end{proof}

\begin{theorem}[Conservation]
The DGC method \eqref{DGCM} preserves all invariants $H_1, H_2, \dots, H_k$ exactly, i.e., $H_i(y_n) = H_i(y_0)$ for all $n$ and $i = 1, \dots, k$.
\end{theorem}

\begin{proof}
Taking the inner product of both sides of the correction scheme
\[
y_{n+1} - \overline{y}_{n+1} = \sum_{i=1}^{k} \lambda_i \overline{\nabla}H_i(\overline{y}_{n+1}, y_{n+1})
\]
with $\overline{\nabla}H_j(\overline{y}_{n+1}, y_{n+1})$ for each $j$, we obtain
\[
\overline{\nabla}H_j(\overline{y}_{n+1}, y_{n+1}) \cdot (y_{n+1} - \overline{y}_{n+1}) = \sum_{i=1}^{k} \lambda_i \overline{\nabla}H_j(\overline{y}_{n+1}, y_{n+1}) \cdot \overline{\nabla}H_i(\overline{y}_{n+1}, y_{n+1}).
\]
By the fundamental property of the discrete gradient yields
\begin{align}\label{4.31}
    H_j(y_{n+1}) - H_j(\overline{y}_{n+1}) = \sum_{i=1}^{k} \lambda_i \overline{\nabla}H_j(\overline{y}_{n+1}, y_{n+1}) \cdot \overline{\nabla}H_i(\overline{y}_{n+1}, y_{n+1}).
\end{align}
On the other hand, the linear system $A\boldsymbol{\lambda} = \mathbf{b}$ implies
\begin{align}\label{4.32}
    \sum_{i=1}^{k} \lambda_i \overline{\nabla}H_j(\overline{y}_{n+1}, y_{n+1}) \cdot \overline{\nabla}H_i(\overline{y}_{n+1}, y_{n+1}) = H_j(y_0) - H_j(\overline{y}_{n+1}).
\end{align}
Substituting \eqref{4.32} into \eqref{4.31} gives
\begin{align*}
    H_j(y_{n+1}) - H_j(\overline{y}_{n+1}) = H_j(y_0) - H_j(\overline{y}_{n+1}),
\end{align*}
which simplifies to 
\begin{align*}
H_j(y_{n+1}) = H_j(y_0).
\end{align*}
The arbitrariness of $j$ implies that the DGC method \eqref{DGCM} can preserve all invariants simultaneously. This completes the proof.
\end{proof}

Algorithm~\ref{algo1} outlines the DGC method \eqref{DGCM} designed to preserve 
$k$ invariants simultaneously, with its applicability naturally extending to the single-invariant case (i.e., $k=1$).
\begin{algorithm}
\caption{DGC method \eqref{DGCM} for preserving multiple invariants}\label{algo1}
\begin{algorithmic}[1]
\State \textbf{for} $n = 0$ \textbf{to} $N-1$ \textbf{do}
\State ~~$\overline{y}_{n+1} \Leftarrow \phi_h(y_n, y_{n-1}, \cdots, y_{n-m})$ \hfill \textit{// prediction step}
\State ~~${y}_{n+1} \Leftarrow \overline{y}_{n+1}$ \hfill \textit{// initial guess}
\State ~~\textbf{repeat}
\State ~~~~Compute $A(\overline{y}_{n+1}, {y}_{n+1})$ and $\mathbf{b}(\overline{y}_{n+1})$
\State ~~~~Solve $\boldsymbol{\lambda} \Leftarrow A^{-1} \mathbf{b}(\overline{y}_{n+1})$
\State ~~~~Update ${y}_{n+1} \Leftarrow g(\overline{y}_{n+1}, {y}_{n+1})$
\State ~~\textbf{until} $\left\| {y}_{n+1} - g(\overline{y}_{n+1}, {y}_{n+1}) \right\| \leq \text{Tolerance}$
\State \textbf{end for}
\end{algorithmic}
\end{algorithm}

\section{Numerical experiments}\label{sec5}
This section evaluates the computational efficacy of the proposed DGC method using representative numerical examples. In practice, we solve the DGC schemes \eqref{DGC1}-\eqref{DGC} and \eqref{DGCM} via fixed-point iteration. The methods used in the prediction step include the forward Euler method and the Runge–Kutta (RK) methods listed in Tables~\ref{tab1} and~\ref{tab2}.

\begin{minipage}{\textwidth}
    \centering
    \begin{minipage}{0.45\textwidth}
        \centering
        \captionof{table}{Third-order
 method $(A,b^1)$ with a
 second-order embedded method
 $(A,b^2)$.}\label{tab1}
        \begin{tabular}{@{}l|lll@{}}
            0 & & & \\
            $1/2$ & $1/2$ & & \\
            1 & -1 & 2 & \\
            \hline
           $b^1$ & $1/6$ & $2/3$ & $1/6$\\
            \hline
            $b^2$&$1/2$&$0$&$1/2$
        \end{tabular}
    \end{minipage}
    \hfill
    \begin{minipage}{0.45\textwidth}
        \centering
        \captionof{table}{Fourth-order
 method $(A,b^1)$ with a
 second-order embedded method
 $(A,b^2)$.}\label{tab2}
        \begin{tabular}{@{}l|llll@{}}
            0 & & & & \\
            $1/2$ & $1/2$ & & & \\
            $1/2$ & 0 & $1/2$ & & \\
            1 & 0 & 0 & 1 & \\
            \hline
           $b^1$ & $1/6$ & $1/3$ & $1/3$ & $1/6$\\
           \hline
           $b^2$ & $1/4$ & $1/4$ & $1/4$ & $1/4$
        \end{tabular}
    \end{minipage}
\end{minipage}

\begin{example}\label{ex:1}
    Consider the Lotka-Volterra model
    \begin{align}\label{LV}
        \begin{cases}
            y'_1 = y_1 (y_2 - 2), \\
            y'_2 = y_2 (1 - y_1),
        \end{cases}
    \end{align}
    which describes the population dynamics of a predator-prey system, where $y_1$ and $y_2$ denote the populations of the predator and prey species, respectively.
    This model possesses a first integral
    \begin{align*}
        H(y_1, y_2) = \ln y_1 - y_1 + 2 \ln y_2 - y_2,
    \end{align*}
    implying that the solution trajectories form closed curves.
\end{example}

We consider the initial value \( \mathbf{y}_0 = (2, 2)^{\top} \) and integrate until \( T = 100 \). Model~\eqref{LV} is solved using the following three methods:
\begin{itemize}
    \item[$\cdot$] Standard forward Euler method.
    \item[$\cdot$] Projected forward Euler method.
    \item[$\cdot$] Discrete gradient correction method (DGC1), which combines the forward Euler method with the coordinate increment discrete gradient~\eqref{DG3}.
\end{itemize}

The numerical results are presented in Table~\ref{tab3}, while the solution trajectories and invariant preservation are depicted in Figs.~\ref{fi1} and~\ref{fi2}.  We observe that:
\begin{itemize}
    \item[$\cdot$] Under relatively large step sizes, the standard forward Euler method exhibits numerical instability, while the DGC1 scheme maintains the theoretical convergence order;
    \item[$\cdot$] At \( h = 2/3 \), Newton's method used in the projected method shows high sensitivity to the initial guess, often leading to convergence failure. In contrast, at the same step size, DGC1 preserves the invariant up to machine precision (see Fig.~\ref{fi2}) and yields the correct solution trajectory;
    \item[$\cdot$] The update direction of the relaxed forward Euler method is tangent to the invariant manifold, resulting in the relaxation equation admitting only the trivial solution $\lambda=0$. As a consequence, the numerical solution remains fixed at the initial state $y_n \equiv y_0$ for all $n$, failing to reproduce the true periodic dynamics.
\end{itemize}

\begin{table}[htbp]
    \centering
    \setlength{\tabcolsep}{6mm}
    \caption{Numerical results for the Lotka-Volterra model.}\label{tab3}
    \begin{threeparttable}
    \begin{tabular}{@{}lllllll@{}}
        \toprule
        \multirow{2}{*}{$h$} & \multicolumn{2}{c}{Euler} & \multicolumn{2}{c}{DGC1} & \multicolumn{2}{c}{Projected} \\
        \cmidrule(lr){2-3}\cmidrule(lr){4-5}\cmidrule(lr){6-7}
        & $L^{\infty}$ Error & Rate & $L^{\infty}$ Error & Rate & $L^{\infty}$ Error & Rate \\
        \midrule
        $2/3$   & *    &    & 2.2255 &   & *   &    \\
        $1/10$  & *    & **    & 1.4102 & 0.6582    & 1.4106  &**    \\
        $1/20$  & *    & **    & 0.4068 & 1.7935 & 0.4064 & 1.7953 \\
        $1/40$  & 21.0124 & **    & 0.1054 & 1.9484 & 0.1053 & 1.9484 \\
        $1/80$  & 5.4956  & 1.9349   & 0.0277 & 1.9279 & 0.0277 & 1.9265 \\
        \bottomrule
    \end{tabular}
     \begin{tablenotes}
\item * Solution failed.
\item ** No convergence rate computed due to preceding failure.
\end{tablenotes}
\end{threeparttable}
\end{table}

\begin{figure}[htbp]
    \centering      \includegraphics[width=0.5\textwidth]{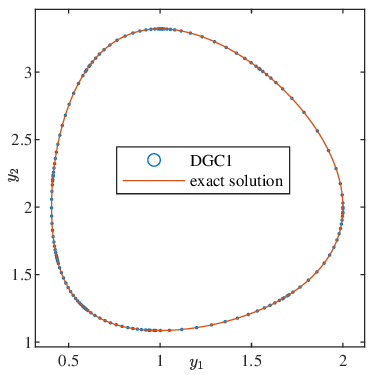}
    \caption{ Lotka-Volterra Model: Exact and Computed (DGC1) Solutions}\label{fi1}
\end{figure}

\begin{figure}[htbp]
    \centering
    \includegraphics{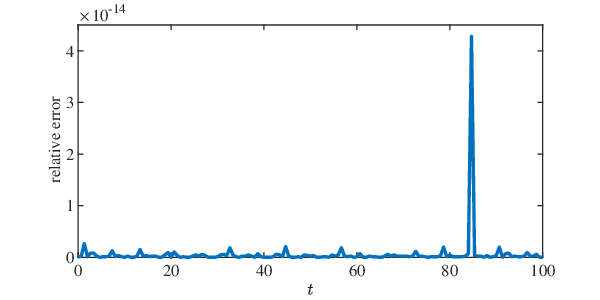}
    \caption{Relative error in the invariant $H$ for the DGC1 method.}\label{fi2}
\end{figure}

\begin{example}\label{ex:2}
    Consider the sine-Gordon equation
    \begin{align}\label{eq:sg}
        \begin{cases}
            u_{tt} - u_{xx} + \sin u = 0, & (x,t) \in [-L, L] \times [0, T], \\
            u(x, 0) = 0, \\
            u_t(x, 0) = 4\kappa\,\mathrm{sech}(\kappa x), \\
            u(L, t) = u(-L, t),
        \end{cases}
    \end{align}
    where $\kappa = \frac{1}{\sqrt{1 + c^2}}$. This system possesses the Hamiltonian
    \[
        H(u(t)) = \frac{1}{2} \int^L_{-L} |u_t|^2 + |u_x|^2 + 2(1 - \cos(u)) \, \mathrm{d}x.
    \]
    The exact solution \cite{Bratsos} is given by
    \[
        u(x, t) = 4 \arctan\left(c^{-1} \sin(c\kappa t) \, \mathrm{sech}(\kappa x)\right).
    \]
    In the following computations, we use the parameters $c = 0.5$ and $L = 20$.
\end{example}

Introducing $v = u_t$, we reformulate the sine-Gordon equation \eqref{eq:sg} as a first-order system:
\begin{align*}
    \begin{cases}
        u_t = v, \\
        v_t = u_{xx} - \sin u.
    \end{cases}
\end{align*}
We discretize space uniformly with $h = 2L/N$ and grid points $x_i = -L + (i-1)h$. Define the solution vectors
\[
    U(t) = [u_1, u_2, \cdots, u_N]^T, \quad V(t) = [v_1, v_2, \cdots, v_N]^T,
\]
where $u_i = u(-L + (i-1)h, t)$ and $v_i = v(-L + (i-1)h, t)$.

We approximate $u_{xx}$ using the spectral differentiation operator
\[
    \mathcal{D} = \mathcal{F}^H \Lambda \mathcal{F},
\]
where $\mathcal{F}$ is the discrete Fourier matrix, $\mathcal{F}^H$ its conjugate transpose, and
\[
    \Lambda = -\left(\frac{2\pi}{2L}\right)^2 \mathrm{diag}\left[0^2, 1^2, \cdots, \left(\frac{N}{2}\right)^2, \left(-\frac{N}{2}+1\right)^2, \cdots, (-2)^2, (-1)^2\right].
\]
For further details on spectral methods, see \cite{ShenJ}. The resulting semi-discrete system is
\begin{align}\label{sg1}
    \begin{cases}
        U_t(t) = V(t), \\
        V_t(t) = \mathcal{D} U(t) - \sin(U(t)).
    \end{cases}
\end{align}
The semi-discrete Hamiltonian is given by
\[
    H(U(t), V(t)) = \frac{h}{2} \left( V(t)^T V(t) - U(t)^T \mathcal{D} U(t) + 2\mathbf{e}^T (\mathbf{e} - \cos(U(t))) \right),
\]
where $\mathbf{e} = [1, 1, \cdots, 1]^T \in \mathbb{R}^N$.

We solve system~\eqref{sg1} with $N = 128$ up to $T = 100$ using the following methods:
\begin{itemize}
    \item[$\cdot$] Discrete gradient correction method combined with the coordinate increment discrete gradient~\eqref{DG3} (DGC2).
    \item[$\cdot$] Projected method.
    \item[$\cdot$] Relaxed Runge-Kutta method (RRK).
    \item[$\cdot$] Average vector field (AVF) method \cite{Celledoni}.
    \item[$\cdot$] Gonzalez discrete gradient method.
\end{itemize}
Here, the DGC, projected, and RRK methods are implemented using the Runge-Kutta schemes defined by $(A,b^1)$ listed in Table~\ref{tab1}.

Table~\ref{tab4} demonstrates that over long-time computations, the DGC2 scheme not only achieves the theoretical order of convergence but also preserves energy exactly. Furthermore, Fig.~\ref{fi3} further illustrates the superior computational efficiency of DGC2 compared to other methods. Fig.~\ref{fi4} shows that during long-time simulation, the DGC2 method successfully captures the breather solution of the sine-Gordon equation, which highlights its reliability in simulating localized nonlinear wave phenomena.

\begin{table}[htbp]
    \centering
    \setlength{\tabcolsep}{8mm}
    \caption{Numerical results for the semi-discrete system~\eqref{sg1} for $T=100$.}\label{tab4}
    \begin{tabular}{@{}cccc@{}}
        \toprule
        $\tau$ & $L^{\infty}$ Error & Rate & Relative Error ($H$) \\
        \midrule
        1/10  & 0.0010      & -    & 1.9600e-15 \\
        1/20  & 7.6908e-05  & 3.7007 & 1.9600e-15 \\
        1/40  & 9.5762e-06  & 3.0056 & 1.9600e-15 \\
        1/80  & 1.2032e-06  & 2.9926 & 8.4000e-16 \\
        \bottomrule
    \end{tabular}
\end{table}

\begin{figure}[htbp]
    \centering
    \includegraphics{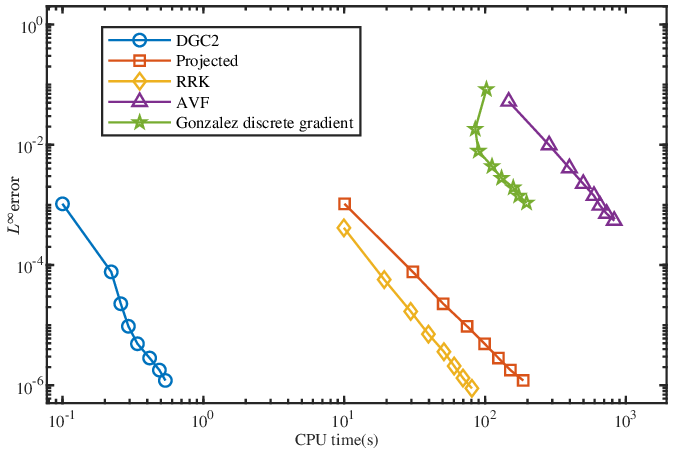}
    \caption{Efficiency curves: $L^{\infty}$ norm errors versus CPU time.}\label{fi3}
\end{figure}

\begin{figure}[htbp]
    \centering
    \begin{subfigure}{0.46\textwidth}
        \centering
        \includegraphics[width=\linewidth]{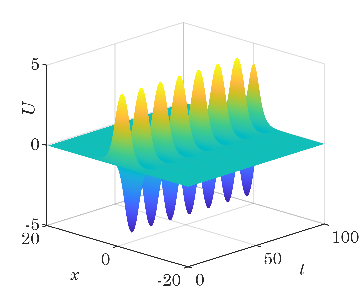}
        \caption{Numerical solution by DGC2 ($\tau = 1/10$).}\label{fi4a}
    \end{subfigure}
    \hfill
    \begin{subfigure}{0.46\textwidth}
        \centering
        \includegraphics[width=\linewidth]{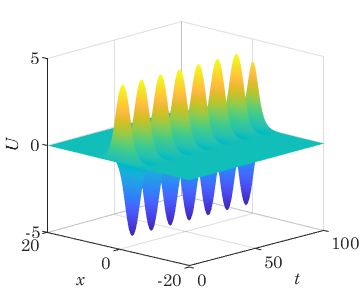}
        \caption{Exact solution ($\tau = 1/10$).}\label{fi4b}
    \end{subfigure}
    \caption{Comparison of numerical and exact solutions for the sine-Gordon equation.}\label{fi4}
\end{figure}
\begin{example}\label{ex:3}
Consider the Euler equations for rigid body dynamics:
\begin{align}\label{exa:3}
    \begin{cases}
        y'_1 = \dfrac{I_2 - I_3}{I_2 I_3} y_2 y_3, \\
        y'_2 = \dfrac{I_3 - I_1}{I_3 I_1} y_3 y_1, \\
        y'_3 = \dfrac{I_1 - I_2}{I_1 I_2} y_1 y_2,
    \end{cases}
\end{align}
where $y = (y_1, y_2, y_3)^\top$ represents the angular momentum vector and $I_i$ are the principal moments of inertia. This system possesses two invariants:
\begin{align*}
    H_1(y_1, y_2, y_3) = \frac{1}{2} \left( \frac{y^2_1}{I_1} + \frac{y^2_2}{I_2} + \frac{y^2_3}{I_3} \right),
\end{align*}
and
\begin{align*}
    H_2(y_1, y_2, y_3) = y^2_1 + y^2_2 + y^2_3.
\end{align*}
\end{example}

We take $I_1 = 2$, $I_2 = 1$, $I_3 = 2/3$, and initial value $y_0 = (\cos(1.1), 0, \sin(1.1))^\top$. The corresponding exact solution can be expressed in terms of Jacobi elliptic functions as:
\begin{align*}
    \begin{cases}
        y_1(t) = \cos(1.1) \, \mathrm{cn} \left( \dfrac{\sin(1.1)}{\sqrt{2}} t, \cot(1.1) \right), \\
        y_2(t) = -\sqrt{2} \cos(1.1) \, \mathrm{sn} \left( \dfrac{\sin(1.1)}{\sqrt{2}} t, \cot(1.1) \right), \\
        y_3(t) = \sin(1.1) \, \mathrm{dn} \left( \dfrac{\sin(1.1)}{\sqrt{2}} t, \cot(1.1) \right).
    \end{cases}
\end{align*}

Setting $T = 1000$, we solve the Euler equations~\eqref{exa:3} using the discrete gradient correction method (DGC3), which combines the Runge-Kutta method with coefficients $(A,b^1)$ from Table~\ref{tab1} with the coordinate increment discrete gradient~\eqref{DG3}. As shown in Table~\ref{tab5}, the DGC3 method preserves both $H_1$ and $H_2$ to machine precision, achieves the expected order of accuracy, and requires only a few iterations per time step to converge. 

Fig.~\ref{fi5} illustrates the numerical trajectories obtained using different conservation strategies with the DGC3 method at $h = 1$, revealing the following observations:
\begin{itemize}
    \item[$\cdot$] Preserving only a single invariant (either $H_1$ or $H_2$) fails to maintain the correct solution trajectory.
    \item[$\cdot$] Simultaneous conservation of both $H_1$ and $H_2$ ensures accurate orbital adherence throughout the integration.
\end{itemize}
This shows that for multi-invariant systems, numerical methods preserving only a subset of invariants may yield unphysical solutions. Fig.~\ref{fi6} demonstrates the superior computational efficiency of DGC3 over the MRRK method based on the coefficients in Table \ref{tab1}.

\begin{table}[htbp]
    \centering
    \setlength{\tabcolsep}{6.5mm}
    \caption{Numerical results for the Euler equations~\eqref{exa:3} at $T = 1000$.}\label{tab5}
    \begin{tabular}{@{}cccccc@{}}
        \toprule
        $h$ & $L^{\infty}$ Error & Rate & $H_1$ Error & $H_2$ Error & Avg. Iter. \\
        \midrule
        1     & 1.1741     & -      & 5.1469e-16 & 3.3307e-16 & 6.0 \\
        1/2   & 0.0979     & 3.5841  & 5.1469e-16 & 4.4409e-16 & 4.3 \\
        1/4   & 0.0061     & 4.0044  & 5.1469e-16 & 4.4409e-16 & 3.5 \\
        1/8   & 3.8334e-04 & 3.9921  & 5.1469e-16 & 4.4409e-16 & 3.0 \\
        \bottomrule
    \end{tabular}
\end{table}

\begin{figure}[htbp]
    \centering
    \begin{subfigure}{0.32\textwidth}
        \centering
        \includegraphics[width=\linewidth]{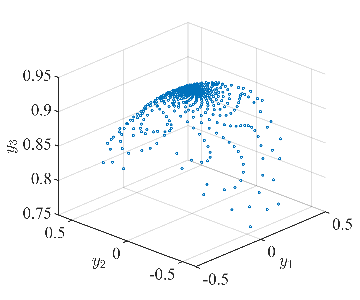}
        \caption{Preserving $H_1$ only}\label{fi5a}
    \end{subfigure}
    \hfill
    \begin{subfigure}{0.32\textwidth}
        \centering
        \includegraphics[width=\linewidth]{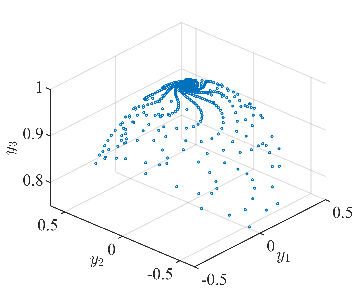}
        \caption{Preserving $H_2$ only}\label{fi5b}
    \end{subfigure}
    \hfill
    \begin{subfigure}{0.32\textwidth}
        \centering
        \includegraphics[width=\linewidth]{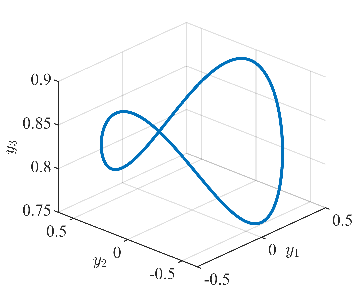}
        \caption{Preserving both invariants}\label{fi5c}
    \end{subfigure}
    \caption{Trajectories under different conservation strategies for the Euler equations with DGC3.}\label{fi5}
\end{figure}
\begin{figure}[htbp]
    \centering
    \includegraphics{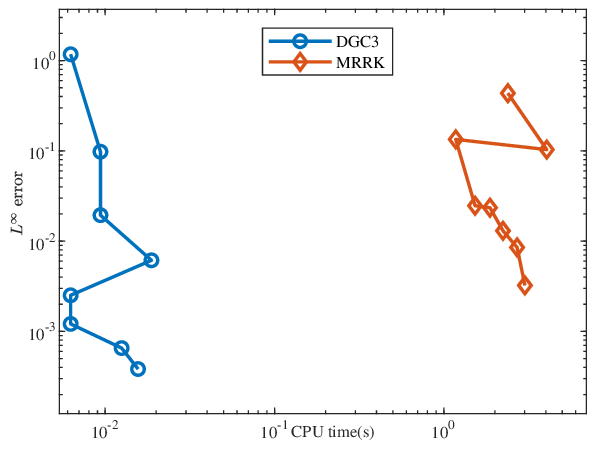}
    \caption{Efficiency curves: $L^{\infty}$ norm errors versus CPU time at $T=1000$.}\label{fi6}
\end{figure}
\begin{example}\label{ex:4}
Consider the Kepler system describing planetary motion:
\begin{align}\label{eq:ex4}
    \begin{bmatrix} p' \\ q' \end{bmatrix} = J \nabla H, \quad 
    J = \begin{bmatrix} 0 & -I \\ I & 0 \end{bmatrix},
\end{align}
with initial conditions
\[
q(0) = (1-e, 0)^\top, \quad p(0) = \left( 0, \sqrt{\frac{1+e}{1-e}} \right)^\top,
\]
where $q = (q_1, q_2)^\top$ denotes position and $p = (p_1, p_2)^\top$ momentum.

The Kepler system possesses two invariants: the Hamiltonian
\[
H(p, q) = \frac{1}{2} \|p\|^2_2 - \frac{1}{\|q\|_2},
\]
and the angular momentum
\[
M = q_1 p_2 - p_1 q_2.
\]
For eccentricity $e \in (0,1)$, the orbiting body follows an elliptical trajectory with the central body at one focus. The exact solution for position is given by
\begin{align*}  
q_1(t) = \cos E - e, \quad q_2(t) = \sqrt{1-e^2} \sin E,  
\end{align*}  
where the eccentric anomaly $E$ satisfies Kepler's equation:
\begin{align*} 
E - e \sin E = t.
\end{align*}
\end{example}

We set the eccentricity to \(e = 0.6\) in our computations.
The Kepler system~\eqref{eq:ex4} is solved using the discrete gradient correction scheme (DGC4), which combines the Runge-Kutta method  with coefficients $(A,b^1)$ from Table~\ref{tab2} with the discrete gradient~\eqref{DG3}. As shown in Table~\ref{tab6}, DGC4 achieves the expected convergence order, preserves both invariants to machine precision, and requires relatively few iterations per time step. 

Fig.~\ref{fi7} shows numerical trajectories obtained with DGC4 at $h = 1/10$ under different conservation strategies. These results demonstrate a phenomenon consistent with Example~\ref{ex:3}: for systems with multiple invariants, preserving only a proper subset may lead to non-physical solutions. 
Fig.~\ref{fi8} provides a direct comparison of computational efficiency. 
It can be observed that the proposed DGC4 method achieves the same level of accuracy faster than the MRRK scheme constructed using the coefficients $(A, b^1, b^2)$ from Table~\ref{tab1}.

\begin{table}[htbp]
    \centering
    \setlength{\tabcolsep}{6mm}
    \caption{Numerical results for the Kepler system at $T = 100$.}\label{tab6}
    \begin{tabular}{@{}cccccc@{}}
        \toprule
        $h$ & $L^{\infty}$ Error & Rate & $H$ Error & $M$ Error & Avg. Iter. \\
        \midrule
        1/10  & 0.0105      & -     & 1.7764e-15 & 4.1633e-16 & 3.0 \\
        1/20  & 9.0552e-04  & 3.5415 & 1.7764e-15 & 4.1633e-16 & 2.5 \\
        1/40  & 6.1083e-05  & 3.8899 & 2.2204e-15 & 4.1633e-16 & 2.2 \\
        1/80  & 3.8972e-06  & 3.9703 & 1.7764e-15 & 4.1633e-16 & 2.0 \\
        \bottomrule
    \end{tabular}
\end{table}

\begin{figure}[htbp]
    \centering
    \begin{subfigure}{0.32\textwidth}
        \centering
        \includegraphics[width=\linewidth]{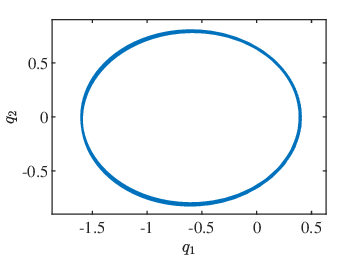}
        \caption{Preserving $H$ only}\label{fi7a}
    \end{subfigure}
    \hfill
    \begin{subfigure}{0.32\textwidth}
        \centering
        \includegraphics[width=\linewidth]{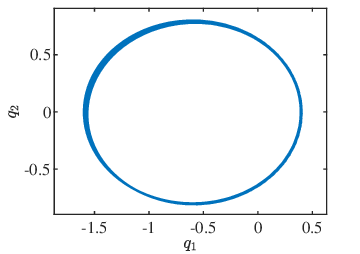}
        \caption{Preserving $M$ only}\label{fi7b}
    \end{subfigure}
    \hfill
    \begin{subfigure}{0.32\textwidth}
        \centering
        \includegraphics[width=\linewidth]{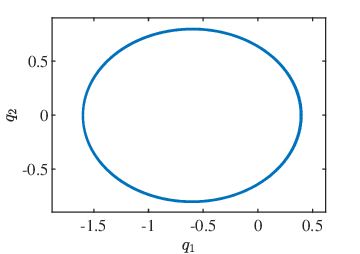}
        \caption{Preserving both invariants}\label{fi7c}
    \end{subfigure}
    \caption{Trajectories under different conservation strategies for the Kepler system with DGC4.}\label{fi7}
\end{figure}
\begin{figure}[htbp]
    \centering
    \includegraphics{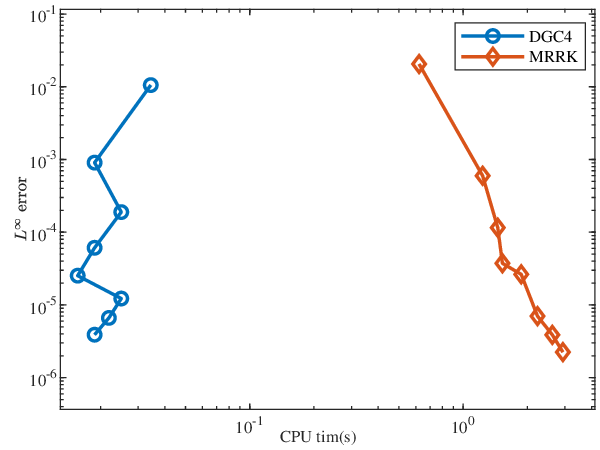}
    \caption{Efficiency curves: $L^{\infty}$ norm errors versus CPU time at $T=100$.}\label{fi8}
\end{figure}
\section{Conclusion}\label{sec6}
The present work introduces the Discrete Gradient Correction (DGC) framework, a predictor-corrector approach that converts arbitrary high-order explicit methods into invariant-preserving integrators through an efficient post-processing step. A central advantage over classical discrete gradient methods is the elimination of the need to construct high-order discrete gradients directly. The proposed framework demonstrates superior robustness compared to projection and relaxed Runge-Kutta methods under large time steps and can be naturally extended to preserve multiple invariants via symmetric linear systems. Theoretical analyses guarantee conservation, order preservation, and stability under mild conditions, with numerical experiments confirming long-term reliability. Application of the DGC method requires known invariant values at each time step. In conservative systems, these values are determined by the initial conditions. Extension to dissipative systems, however, is not straightforward, as the corresponding quantities are no longer conserved. Adapting the framework for such cases, potentially through estimation of time-dependent targets or a reformulated correction procedure, poses a challenge for future research.

\section*{Acknowledgements}
This work is supported by the Sichuan Science and Technology Programs (No. 2024NSFSC0441).

\section*{Declarations}
\textbf{Conflict of interest} The authors have no competing interests to declare that are relevant to the content of this article.

\bibliographystyle{spmpsci}
\bibliography{sn-bibliography}

\end{document}